\documentclass[11pt]{amsart}
\usepackage[T1]{fontenc}
\usepackage[latin9]{inputenc}

\usepackage{babel}
\usepackage{amsmath, amssymb}
\usepackage{tikz-cd}
\usetikzlibrary{cd}

\newtheorem{Theorem}{Theorem}[section]

\newtheorem{Lemma}[Theorem]{Lemma}
\newtheorem{Proposition}[Theorem]{Proposition}

\newtheorem{Definition}[Theorem]{Definition}
\newtheorem{rem}[Theorem]{Remark}
\newtheorem{example}[Theorem]{Example}

\newcommand{\R}{\mathbb R}

\newcommand{\N}{\mathbb N}

\begin{document}
\title{On random pairwise comparisons matrices and their geometry.}
\author{Jean-Pierre Magnot}
\address{CNRS, LAREMA, SFR MATHSTIC, F-49000 Angers, France \\and\\ Lyc\'ee Jeanne d'Arc, Avenue de Grande Bretagne, F-63000 Clermont-Ferrand}
	\email{jean-pierr.magnot@ac-clermont.fr}

\maketitle
\begin{abstract}
We describe a framework for random pairwise comparisons matrices, inspired by selected constructions releted to the so called inconsistency reduction of pairwise comparisons (PC) matrices. In to build up structures on random pairwise comparisons matrices, the set up for (deterministic) PC matrices for non-reciprocal PC matrices is completed. The extension of basic concepts such as inconsistency indexes and geometric mean method are extended to random pairwise comparisons matrices and completed by new notions which seem useful to us. Two procedures for (random) inconsistency reduction are sketched, based on well-known existing objects, and a fiber bundle-like decomposition of random pairwise comparisons is proposed. 
\end{abstract}

\tableofcontents
\section*{Introduction}
The wide popularity of the pairwise comparison methods in the field of multi-
criteria decision analysis mostly due to their simplicity and their adaptability. Since it is  a natural approach to compair two by two the items under control, it is not surprising that the first systematic use of pairwise comparisons is
attributed to Ramon Llull \cite{6} during the thirteenth-century.

Very probabily, people interested in comparisons passed from qualitative to quantitative evaluations for the obvious gan in simplicity of interpretation. The twentieth-century
precursor of the quantitative use of pairwise comparison was Thurstone \cite{43} and
the pairwise comparison method has been progressively improved till the groundbreaking work of Saaty \cite{S1977}. In this article Saaty proposed the Analytic Hierarchy Process
(AHP), a new multiple-criteria decision-making method. Thanks to the popularity of AHP the pairwise comparison method has become one of the most
frequently used decision-making techniques. Its numerous applications include most existing fields such as economy (see e.g.\cite{37}), consumer research \cite{14}, management \cite{45}, military science \cite{CW1985,18}, education and science
(see e.g. \cite{32}) among others.

Besides its great popularity, its wide and intensive application in everyday's life highlights tremendouly its weakness with respect to the following to aspects: 
\begin{itemize} 
\item Popularity of the decision-making methods also makes them vulnerable to at-
tacks and manipulations. This problem has been studied by several researchers
including 
Sasaki \cite{41}, on strategic manipulation in the context of group decisions with
pairwise comparisons. Some aspects of decision manipulation in the context of
electoral systems are presented in e.g. \cite{42}.
\item Even in honestly made evaluations, a small change in the scores in the pairwise comparisons \cite{SKE2023} or of the secondary criteria such as the so-called inconsistency index \cite{MMC2023} can change drastically the final result of the ranking procedure.
\end{itemize}
We have recently linked explicitely inconsistency in pairwise comparisons (PC) with Yang-Mills theories, showing a possible (non-Abelian) generalization of the classical framework \cite{Ma2018-1,Ma2018-3}. Our approach is not the only one where pairwise comparisons coefficients are not numerals, see e.g. \cite{BR2022,KSW2016,37,W2019}, but the link with existing lattice-based or distretized gauge theories appears to us very promising, not only because of the very wide literature on the subject, but also because gauge theories also seem to be promising candidates to furnish applications out of the field of physics, see e.g. the amazing \cite{Mac95}, or \cite{S44} in Neuronal network and \cite{FV2012,I2000,Y1999} in finance. We have here to precise that discretization of gauge theories (in order to recover a continuum limit) are very often treated under the lights of Whitney approximations \cite{Wh} of differential forms and connections on trivial principal bundles, see e.g. \cite{AZ1990}, and very rarely from the viewpoint of the discretization commonly proposed by quantum gravity \cite{RV2014} based on the computation of the holonomies (in the physics sense) of the connections along the 1-vertices of the triangulation (see \cite{Ma2018-2} for a rigorous description of the discretization procedure). These are holonomy-like discretized connections that we consider in this work, that can therefore be identified with  function from oriented 1-vertices with values in the structure group, that satisfy so-called reciprocity conditions.

 In this communication, we intend to develop a theory of random pairwise comparisons matrices that was shortly sketched at the end of \cite{Ma2018-3}, and describe the structures of the proposed sets, in order to extend the notions of weight, consistency, inconsistency indicator, in a stochastic context focusing on the very usual choice of $\R_+^*$ as a structure group. Inspired by the correspondence between $\R_+^*-$gauge theories in quantum gravity and classical pairwise comparisons (PC) matrices highlighted in \cite{Ma2018-3}, we follow up this first approach by replacing a pairwise comparisons matrix by a probability on the space of pairwise comparisons matrices, understood as a random variable, in the same spirit as the statistical mechanics replace a particule with fixed position by a probability. In other words, we propose here a first quantization of pairwise comparisons.  This seems necessary to us in view of the variations, dishonest or not, that it is important to modelize in order to have a better control on them as explained before. For this, the necessity of the use of non-reciprocal pairwise comparisons matrices, an extended class of pairwise comparisons that may not follow the logic rule of reciprocity, also appears to us as mandatory.
 
 Therefore, after an oriented summary of the existing framework (of interest for us) in this direction, mostly based on \cite{Ma2018-3}, in order to define safely random pairwise comparisons matrices, we extend the geometry of classical PC matrices understood as an abelian group, to non reciprocal PC matrices which project to the total tangent space of the first group. This enables us to extend the notion of consistency to non-reciprocal (deterministic) PC matrices and define canonical projections on remarkable subsets, inconsistency indicators, reciprocity indicators and so-called total indexes that ``gather'' the two criteria, consistency and reciprocity, in their evaluation. With these deterministic tools at hand, we define their stochatic analogs on random PC matrices and stochastic projections on remarkable subsets of random matrices. Meaningful decompositions of the space of random matrices are described, and two complementary stochastic methods are proposed for reduction of stochastic non-reciprocity, stochastic inconsistency, with in mind the problem of finding out weights for a final evaluation of the items at hand.

\section{A not-so-short introduction to pairwise comparisons as a class of abelian Yang-Mills theory} \label{2.1.}
A paiswise comparisons matrix $(a_{i,j})$ is a $n \times n$ matrix with coefficients in $\R_+^* = \{ x \in \R \, | \, x>0\}$ such that $$\forall i,j, \quad a_{j,i} = a_{i,j}^{-1}.$$
It is easy to explain the inconsistency in pairwise comparisons when we consider cycles of three comparisons, called triad and represented here as $(x,y,z)$,
which do not have the ``morphism of groupoid'' property such as $$x * z \neq y,$$ which reads as $$ xz \neq y $$ in the multiplicative group $\R_+^*.$ 
\noindent The use of ``inconsistency'' has a meaning of a measure of inconsistency in this study; not the concept itself. 
An approach to inconsistency (originated  in \cite{K1993}) can be reduced to a simple observation:
\begin{itemize}
	\item 
	search all triads (which generate all 3 by 3 PC sub matrices) and locate the worst triad with a so-called inconsistency indicator ($ii$),
	\item 
	$ii$ of the worst triad becomes $ii$ of the entire PC matrix.
\end{itemize}

\noindent Expressing it a bit more formally in terms of triads (the upper triangle of a PC sub matrix $3 \times 3$), we have \emph{Koczkodaj's index}, see e.g. \cite{BR2008,KS2014}:
\begin{equation} \label{kii3}Kii(x,y,z) = 1-\min\left \{\frac{y}{xz},\frac{xz}{y}\right \}=1- e^{-\left|\ln\left (\frac{y}{xz}\right )\right |}. \end{equation}
  For the sake of completeness, we notice here that this definition allows us to localize the inconsistency in the matrix PC and it is of a considerable importance for most applications. But this aspect will be worked out in future works, and we prefer here to concentrate of reduction of the inconsistency.
This inconsistency indicator extends to $n \times n$ PC matrices by:
\begin{equation} \label{Kn}
	{ Kii}_n(A)=\max_{(i,j)\in \N_n^2, i<j} \left(1 - e^{-\left|\ln \left( \frac{a_{ij}}{a_{i,i+1}a_{i+1,i+2}\ldots  a_{j-1,j}}  \right)\right| }\right)
\end{equation}
where $\N_n = \{1,\cdots,n\} \subset \N,$
but also by various ways to combine inconsistency of triads, see e.g. the examples studied in the recent work \cite{MMC2023}.

 Another common procedure on pairwise comparisons matrices changes the multiplicative group $(\R^*_+,.)$ to the additive group $(\R,+)$ with the same (obvious) adaptations of the formulas. These transformations, natural from a group-theoric viewpoint and in one-to-one correspondence via logarithmic maps, are actually perceived as no-so natural in the framework of decision theory. Out of psychological reasons, this basically comes from the fact that in the framework of pairwise comparisons, the important procedure consists in finding an efficient way to reduce an inconsistency indicator, which may not be $Kii_3$ or $Kii_n$ but another, which is chosen because more adapted to the contextual problems. For this, one defines \textbf{priority vectors} which intend to give a ``best way'' to reduce in a satisfying manner the inconsistency, that is, the ``score'' of inconsistency produced by the inconsistency index. Therefore, the correspondence between additive and multiplicative pairwise comparisons matrices enables one to produce more simple formulas and algorithms in the additive case.
 This correspondence is the key motivation for e.g. \cite{KSS2020} where an apparently natural way to link any PC matrix with a ``prefered'' consistent one is described. This impression of simplicity and clarity is obtained by the use of basic geometrical methods, mostly the use of orthogonal vector spaces. However, in the actual state of knowledge, the choice any method for inconsistency reduction appears as arbitrary from the viewpoint of applications while almost canonical from the viewpoint of mathematical structures, because the procedure of reduction of inconsistency may produce very far pairwise comparisons matrices that can be controlled by a small enough inconsistency indicator value \cite{MMC2023}. 
 
  Following \cite{Ma2018-1,Ma2018-3}, let $n$ be the dimension of the PC matrices under consideration, and let us consider the $n-$simplex
$$\Delta_n = \Big\{ (x_0,\ldots, x_n) \in \mathbb{R}^{n+1} |$$
$$ \left(\sum_{i=0}^n x_i = 1\right) \wedge \left(\forall i \in \{0,\ldots.n\}, x_i \geq 0\right)\Big\}$$ as a graph, made of $\frac{n(n-1)}{2}$ edges linking $n$ vertices, equipped with a pre-fixed numerotation. Then there is a one-to-one and onto correspondence between edges and the positions of the coefficients in the PC matrix. Hence, each PC matrix $(a_{i,j})$ assigns the coefficient $a_{i,j}$ to the (oriented) edge from the $i-$th vertex to the $j-$th vertex. These are the \textbf{holonomies} of a connection, understood in its discretized analog (in a discretization throuhg holonomy around the oriented 1-vertices of the simplex). Minimizing holonomies is a quantum gravity analog of a classical $\R_+^*-$Yang-Mills theory \cite{RV2014,Seng2011}. Indeed, in the quantum gravity approach \cite{RV2014}, the seek of flatness of the curvature relies in the  minimization of the distance between the loop holonomies and $1,$ which is exactly the basic requirement of reduction of inconsistencies on triads. Following the preliminary note \cite{Ma2018-1}, there is one conceptual problem in the correspondence between Koczkodaj's inconsistency index and the Yang-Mills functional, highlighted from a psychological viewpoint in \cite{MMC2023}.
Koczkodaj's inconsistency indicators read as: 
 $${ Kii}_n(A)=$$
 $$1 - exp \left(-\max_{1\le i<j\le n}\left |\ln \left ( {a_{ij}\over
 		a_{i,i+1}a_{i+1,i+2}\ldots  a_{j-1,j}}  \right )\right |\right).$$
 	In order to recover now the quamtum gravity analog \cite{RV2014} of the Yang-Mills action functional, one replaces $$\max_{1\le i<j\le n}(...)$$ by a quadratic mean. The choice of an inconsistency indicator is not so easy and Koczdodaj's, even if quite simple and easy to understand, cannot be such easily considered as the only interesting indicator for inconsistency in pairwise comparisons, see e.g. Saaty's \cite{BR2008,S1977}.

 	These are all these relations that we now reinforce in this work, by analyzing non-reciprocal pairwise comparisons matrices from this viewpoint with the help of complementary notions. 
 	\section{A theory of non reciprocal pairwise comparisons and weights: deterministic settings}
 	In the rest of the text, $n$ is the dimension of the square matrices under consideration.
 	\subsection{Preliminary settings on $M_n(\R_+^*)$ and its remarkable subsets}
 	We consider here $M_n(\R_+^*)$ ans a smooth manifold of dimension $n^2$ with tangent space $M_n(\R).$ We also define $G_n = (\R_+^*)^n$ (this notation will be explained later). We recall that we note by $PC_n$ the subset of $M_n(\R_+^*)$ made of matrices $(a_{i,j})_{(i,j)\in \N_n^2}$ such that $$\forall (i,j) \in \N_n^2, a_{i,j}=a_{j,i}^{-1}$$ (which implies that $\forall i \in \N_n, \, a_{i,i}=1$)
 	We define the following operations:
 	\begin{enumerate}
 		\item $M_n(\R_+^*)$ is stable for the matrix multiplication in $M_n(\R),$ noted by  $$(M,N) \mapsto MN$$ and we choose to keep the same notation for its restriction to  $M_n(\R_+^*).$
 		\item The group multiplication in $\R_+^*$ extends to $M_n(\R_+^*)$ by identifying the matrices of   $M_n(\R_+^*)$ with its coefficients, that is we make the identification $M_n(\R_+^*) = (\R_+^*)^{n^2}.$ We therefore define a free Lie group structure on $M_n(\R_+^*)$ that we note by $(A,B) \in M_n(\R_+^*)^2 \mapsto C \in M_n(\R_+^*)$ where $C=(c_{i,j})_{(i,j)\in \N_n^2}$ and $$\forall (i,j) \in \N_n^2, \quad c_{i,j} = a_{i,j}b_{i,j}$$ with $A=(a_{i,j})_{(i,j)\in \N_n^2}$ and $B=(b_{i,j})_{(i,j)\in \N_n^2}.$  $PC_n$ and $CPC_n$ are Lie subgroups of $M_n(\R_+^*)$ equipped with this multiplication. 
 		\item The same way, reformulating and extending \cite{Ma2018-3}, $*$ can be also applied to $G_n$ and there are three smooth actions of $G_n$ on $M_n(\R_+^*)$ noted by $L, R$ and $Ad$ and, noting $G=(g_{k})_{k \in \N_n} \in G_n$ and $A=(a_{i,j})_{(i,j)\in \N_n^2} \in M_n(\R^*_+),$ they are defined respectively by
 		$$ L: (G,A) \mapsto (g_i a_{i,j})_{(i,j)\in \N_n^2}$$ if $i<j$ and $$ L: (G,A) \mapsto (g_i^{-1} a_{i,j})_{(i,j)\in \N_n^2}$$ if $i>j,$
 		$$ R: (G,A) \mapsto (g_j a_{i,j})_{(i,j)\in \N_n^2}$$if $i<j$ and $$ R: (G,A) \mapsto (g_j^{-1} a_{i,j})_{(i,j)\in \N_n^2}$$ if $i>j,$
 		and $$Ad (G,A) = L(G,R(G^{-1},A)).$$ The last formula highlights the link of $G$ with the gauge group of a discretized connection as described in \cite{Ma2018-3} in a study on $PC_n$ instead of $M_n(\R^*_+).$ One can remark that the actions $L$ and $R$ correspond to $*-$multiplicaton by matrices with constant lines (resp. constant columns).
 		\item Let $\mathfrak{G}_n$ be the group of $n-$permutations understood as the group of bijections of $\N_n.$ Then $\mathfrak{G}_n$ acts indexwise (and hence smoothly) on $M_n(\R^*_+)$ and the submanifolds $PC_n$ and $CPC_n$ are stable under this action. 
 	\end{enumerate} 
 	\subsection{Weights as projective coordinates}
 	We recall that family of weights $(w_i)_{i \in \N_n} \in (\R^*_+)^n$ is relative to a consistent PC matrix $A = (a_{i,j})_{(i,j)\in \N_n^2} \in  CPC_n(\R_+^*)$ if and only if $$\forall (i,j)\in \N_n^2, a_{i,j}= \frac{w_j}{w_i}.$$
 	Therefore, we have another interpretation of weights of a consistent pairwise comparisons matrix:
 	\begin{Proposition}
 		There exists a smooth embedding $\phi$ from $CPC_n(\R^*)$ to $\mathbb{P}_n(\R)$ such that, for each $A = (a_{i,j})_{(i,j)\in \N_n^2}\in CPC_n,$  $(w_i)_{i \in \N_n} \subset (\R_+^*)^n$ are projective coordinates for $\phi(A)$ if and only if $$a_{i,j} = w_i / w_j.$$
 	\end{Proposition}
 The proof of the fact that $\phi$ is an embedding is obvious.
 \begin{Proposition}
 	$\phi(CPC_n)$ is open and dense in $\mathbb{P}_n(\R)$ and the map $\phi$ is proper.
 \end{Proposition}
\begin{proof} $dim CPC_n = dim \mathbb{P}_n(\R) = n-1$ and $\phi$ is an embedding therefore $\phi(A)$ is an open submanifold of $\mathbb{P}_n(\R)$ and $\phi$ is a diffeomorphism from $CPC_n$ to $\phi(CPC_n),$ which implies that $\phi$ is proper. Moreover, if we consider the canonical embaddings of $\mathbb{P}_{n-1}(\R)$ into $\mathbb{P}_n(\R),$ defined by sending $(n-1)$ homogeneous coordinates into $n$ homogeneous coordinates after fixing one of them equal to $0,$ the set $$\mathbb{P}_n(\R) - \phi(CPC_n)$$ is composed of the images of these $n$ embeddings of $\mathbb{P}_{n-1}(\R)$ into $\mathbb{P}_n(\R),$ and, because each of them is of dimension $n-2 < n-1 = dim \mathbb{P}_n(\R),$ each of them has an empty interior. Therefore $\phi(CPC_n)$ is dense in $\mathbb{P}_n(\R).$ 
	\end{proof}
 	\subsection{(Reciprocal) pairwise comparisons}
 		We recall that a pairwise comparisons matrix is a matrix $A \in M_n(\R_+^*),$ such that $$\forall (i,j)\in \N_n^2, \quad a_{j,i} = a_{i,j}^{-1}.$$
 		This condition is called the reciprocity condition, and the word ``reciprocal'' is often omitted in the literature. However, since we consider here a wider class of pairwise comparisons, that can fulfill the reciprocity condition or not, we can use the more precise terminology ``reciprocal pairwise comparisons'' when we feel a possible ambiguity in the text that could lead to misleading statements.  Following \cite{Ma2018-3}, $PC_n$ is a smooth manifold of dimension $\frac{n(n-1)}{2}.$ As a free abelian Lie group, it is a Riemannian exponential Lie group, as well as $CPC_n.$
 		Moreover, 
 		\begin{Theorem} \cite{Ma2018-3}
 			$L(G_3,CPC_3)=PC_3$ and $\forall n \geq 4, L(G_n,CPC_n) \neq PC_n.$ 
 		\end{Theorem}
 		For application purposes, one not only needs to decide if a pairwise comparisons matrix is consistent, but also if an inconsistent pairwise comparisons matrix is ``very'' inconsistent or ``not so much'' inconsistent. This leads to the definition of the following objects.
 		\begin{Definition} \label{d:ii}
 			An \textbf{inconsistency index} is a map $ii:PC_n \rightarrow \R_+$ such that $$ii^{-1}(0) = CPC_n.$$
 			An \textbf{inconsistency indicator} is an inconsistency index with values in $[0,1].$
 		\end{Definition}   
 	\begin{rem}
 		This terminology is common in the literature, and it is sometimes applied to maps for which $ii^{-1}(0)$ also contains non consistent pairwise comparisons matrices. This ambiguity is due, to our feeling and to our knowledge, mostly to the development of the theory with very popular maps, used historically to detect submatrices with the ``highest'' possible inconsistency. There are propositions on the conditions that must fulfill an inconsistency index, which are not all equivalent, but most of them require that $ii^{-1}(0) = CPC_n.$ This work is not the place for controversies on this subject, and we decide to work in the context here announced.
 		
 		The same way, the actual state of knowledge do not produce a shared requirement for the regularity of the map $ii.$ Does it need to be continuous, derivable or smooth? There are divergent opinions due to the various fields of interest of the authors and, again, to the historical development of the theory. For example, one of the most famous scales for inconsistency envolves natural numbers from $0$ to $9$ (Saaty's scale). The regularity that we assume for an inconsistency index will be precised in the statements.
 	\end{rem}
 The family $$\left\{ii^{-1}([0,\epsilon)) \, | \, \epsilon \in \R_+^* \right\}$$ forms a base of filter around $CPC_n$ \cite{Ma2018-3}. In \cite{MMC2023}, inconsistency indexes of sufficient regularity are treated as scalar potentials on $PC_n$ that have to be minimized by a gradient method. 
 
 	\subsection{Non-reciprocal pairwise comparisons and related indicators}
 	
 		We recall that a non-reciprocal pairwise comparisons matrix is a matrix $A \in M_n(\R_+^*),$ without any further assumption. In particular, if $A$ is not reciprocal, $$\exists (i,j)\in \N_n^2, \quad a_{i,j} \neq a_{j,i}^{-1},$$ and more generally in our terminology, a non-reciprocal matrix may be reciprocal or not reciprocal. We can then adapt the following definition from the notion of inconsistency index:
 		
 			\begin{Definition}
 			A \textbf{reciprocity index} is a map $ri:M_n(\R_+^*) \rightarrow \R_+$ such that $$ri^{-1}(0) = PC_n.$$
 			A \textbf{reciprocity indicator} is an reciprocity index with values in $[0,1].$
 		\end{Definition}   
 	\begin{rem}
 		In order to be coherent with the terminology of Definition \ref{d:ii}, we should prefer the terminology ''non-reciprocity index'' in the last definition. However, since this notion is new, we feel that it is subject to adaptations in future developments and hence the terminology will be fixed naturally by practice.
 	\end{rem}
   In the same spirit, we have
  
  	\begin{Definition}
  	An \textbf{inconsistency and reciprocity index} (or \textbf{total index}) is a map $iri:M_n(\R_+^*) \rightarrow \R_+$ such that $$iri^{-1}(0) = CPC_n.$$
  	An \textbf{inconsistency and reciprocity indicator} (or \textbf{total indicator}) is an inconsistency and reciprocity index with values in $[0,1].$
  \end{Definition}
 Let us give now few classes of examples: 
 \begin{example}
 	Let $d$ be the canonical Riemannian metric on the Lie group $(M_n(\R_+^*),*)$ Then, for any $\gamma \in \R_+^*,$
 	$$ A \in M_n(\R_+^*) \mapsto d(A,PC_n)^\gamma$$
 	is a reciprocity index and 
 	$$ A \in M_n(\R_+^*) \mapsto d(A,CPC_n)^\gamma$$
 	is an inconsistency and reciprocity index. 
 	Moreover, if we replace $d$ by another distance, we get another inconsistency index, and if we moreover assume that the diameter of $M_n(\R_+^*)$ is 1, we get an inconsistency indicator. This class of indicators is different, but may be treated in the same spirit, as those taken as case studies in \cite{MMC2023}. 
 \end{example}
 
 \subsection{On the geometry of the tower $CPC_n \subset PC_n \subset M_n(\R_+^*).$}
 Indeed, $M_n(\R_+^*)$ completes the tower of Riemannian abelian Lie groups
 $$M_n(\R_+^*) \supset PC_n \supset CPC_n$$
 and we now feel the need to investigate intrinsic geometric properties of these inclusions. 
 Let us in particular study the inclusion $PC_n \subset M_n(\R_+^*).$ 
 
 \begin{Proposition}
 	$M_n(\R)$ is the (abelian) Lie algebra of $M_n(\R_+^*)$ and the Lie algebra $\mathfrak{pc}_n \subset M_n(\R)$ of $PC_n$ is the vector space of skew-symmetric matrices with coefficients in $\R.$
 \end{Proposition}
 
  \begin{proof}
  Since $M_n(\R_+^*)$ and $PC_n$ are abelian Lie groups, the Lie bracket on their Lie algeba is vanishing. Moreover, since multiplication is coefficientwise, its is trivial that $M_n(\R)$ is the Lie algebra of $M_n(\R_+^*).$ 
  
  Finally, since  $(a_{i,j})_{(i,j)\in \N_n^2} \in M_n(\R_+^*)$ is reciprocal if and only if $$\forall (i,j)\in \N_n^2, \quad a_{i,j}=a_{j,i}^{-1},$$ we get, by differentiation of this constraint at $a_{i,j} =  1$, that $(a_{i,j})_{(i,j)\in \N_n^2} \in M_n(\R)$ is moreover in $\mathfrak{pc}_n$ if and only if $$\forall (i,j)\in \N_n^2, \quad b_{i,j}=-b_{j,i}.$$
  	\end{proof}
  	
 Let us now split $M_n(\R_+^*)$ in order to recover, in some way that we will soon precise, a reciprocal and a non-reciprocal part. Before that, we feel the need to recall the geometric mean method which projects $PC_n$ to $CPC_n$ through the projection $\pi_{GMM}$ defined by
 $$ \pi_{GMM}\left((a_{i,j})_{(i,j)\in \N_n^2}\right) = (b_{i,j})_{(i,j)\in \N_n^2}$$
 with 
 $$\forall (i,j)\in \N_n^2, \quad b_{i,j} = \sqrt[n]{\frac{\prod_{k=1}^n a_{i,k}}{\prod_{k=1}^n a_{l,j}}}.$$	
 \begin{rem} This method is also described in the case of \emph{additive} pairwise comparisons matrices \cite{KSS2020} but it is a simple fact that additive and multiplicative pairwise comparisons are in one to one correspondence, see e.g. the brief but very clear exposition in \cite{SKE2023}. From a geometric viewpoint, additive pairwise comparisons form the Lie algebra of the group of multiplicative pairwise comparisons, which is trivially an exponential group. This easy fact will be of central interest very soon in our work.
 	\end{rem}
 There is a straightforward projection $\pi_{PC}:M_n(\R_+^*)\rightarrow PC_n$ defined by
 $$\pi_{PC}\left((a_{i,j})_{(i,j)\in \N_n^2}\right) = (b_{i,j})_{(i,j)\in \N_n^2}$$
 with 
 $$\forall (i,j)\in \N_n^2, \quad b_{i,j} = \sqrt{a_{i,j}/a_{j,i}}.$$
 
 \begin{rem}
 	If we analyze the diagonal coefficients of the image we get trivially that $\forall i\in \N_n, \, b_{i,i} = 1$
 \end{rem}
 Therefore, there is another natural projection that we call $\pi_{NR}$ defined by 
  
   $$\pi_{PC}\left((a_{i,j})_{(i,j)\in \N_n^2}\right) = (b_{i,j})_{(i,j)\in \N_n^2}$$
  with 
  $$\forall (i,j)\in \N_n^2, \quad b_{i,j} = \sqrt{a_{i,j}a_{j,i}}.$$
  Next proposition is straigthforward:
  \begin{Proposition} \label{prop:dec}
  	$\forall A \in M_n(\R_+^*), A = \pi_{PC}(A)* \pi_{NR}(A),$ and  
  	\begin{itemize}
  		\item if $A = (a_{i,j})_{(i,j)\in \N_n^2},$ $ \pi_{PC}(A)=\pi_{NR}(A) \Leftrightarrow \forall (i,j)\in \N_n^2, a_{i,j} = 1$
  		\item $\pi_{PC}(M_n(\R_+^*)) \cap \pi_{NR}(M_n(\R_+^*))$ is uniquely composed by the matrix $A = (a_{i,j})_{(i,j)\in \N_n^2},$ such that $\forall (i,j)\in \N_n^2, \quad a_{i,j} = 1.$
  	\end{itemize}
  	In other words, the Lie group $M_n(\R_+^*)$ decomposes into a direct product of $\pi_{PC}(A)$ and of $\pi_{NR}(A).$ Moreover, projections are Lie group morphisms.
  \end{Proposition}
  \begin{Lemma} \label{dim NR}
  	$dim\left(\pi_{NR}(M_n(\R_+^*))\right) = \frac{n(n+1)}{2}.$
  \end{Lemma}
  \begin{proof} A matrix in  $\pi_{NR}(M_n(\R_+^*))$ is symmetric in its off-diagonal coefficients, but it has its $n$ diagonal coefficients that may be different from $1.$ Therefore, the dimension of this Lie group is $\frac{n(n+1)}{2} + n = \frac{n(n+1)}{2}.$
  	\end{proof}
  	\begin{proof}[Proof of Proposition \ref{prop:dec}] Through direct calculations, the multiplication map $$ * :\pi_{PC}(M_n(\R_+^*)) \times \pi_{NR}(M_n(\R_+^*)) \rightarrow  M_n(\R_+^*)$$ is a morphism of connected Lie groups of the same dimension with trivial kernel, with $Im\pi_{NR} = Ker \pi_{PC}$ and $Im \pi_{PC} = Ker \pi_{NR}.$
  	\end{proof}
 Let us now investigate a relationship with discretized connections that extends the existing identifications in $PC_n.$ For this, we first remark that  Lemma \ref{dim NR} shows that the decomposition is not exactly a doubling of the structure group from $PC_n$ to $M_n(\R_+^*)$ because $dim\left(\pi_{NR}(M_n(\R_+^*))\right) > dim\left(PC_n\right).$ Roughly speaking, the dimensions that we have ``in more'' in $\pi_{NR}(M_n(\R_+^*)$ are the diagonal coefficients of the matrices. Therefore, we propose a new structure group $$ G = T\R_+^* = \R_+^* \times \R.$$ 
and a family of smooth maps $\phi_\alpha: (\R_+^*)^2 \rightarrow \R_+^* \times \R$ indexed by $\alpha \in \{1,0,-1\}$  defined by 
$$\phi_\alpha(a,b) = \left\{\begin{array}{ccl} (a, \alpha \log b) & \hbox{ if } & \alpha \neq 0 \\
(1,\log b) &\hbox{ if } & \alpha = 0\end{array}\right.$$
Let $\epsilon(i,j)$ be the sign of $i-j$ for $(i,j) \in \N_n^2.$
\begin{Definition}
	For any $A \in M_n(\R_+),$ such that $\pi_{PC}(A)= (c_{i,j})_{(i,j)\in \N_n^2} $ and $\pi_{NR}(A) = (d_{i,j})_{(i,j)\in \N_n^2} ,$ we define the matrix $\Phi(A) = (b_{i,j})_{(i,j)\in \N_n^2}$ with coefficients in $G$ by 
	$$\forall (i,j)\in \N_n^2, b_{i,j} = \phi_{\epsilon(i,j)}(c_{i,j},d_{i,j}).$$
\end{Definition}
For the sake of the clarity of next theorem, we note by $PC_n(\R_+^*)$ the group of multiplicative pairwise comparisons matrices with coefficients in $\R_+^*$ and by $PC_n(\R)$ the group of additive paiwise comparisons matrices with coefficients in $\R.$ 
\begin{Theorem}
	$\Phi(M_n(\R_+))$ is identified as a Lie group with $PC_n(\R_+^*) \times PC_n(\R) \times \R^n.$ With the notations of the last definition, the identification is given by 
	$$b_{i,j} \mapsto (c_{i,j}, \epsilon(i,j) d_{i,j} )\in PC_n(\R_+^*) \times PC_n(\R) $$
	and $$A \mapsto (d_{i,i})_{i \in \N_n} \in \R^n.$$
\end{Theorem}
The proof is straightforward.
\begin{rem} In \cite{Ma2018-3}, for any Lie group $G,$ we described pairwise comparisons matrices $PC_n(G).$ We have here $$PC_n(\R_+^*) \times PC_n(\R) = PC_n(G)$$ for the struture group already announced $$ G = T\R_+^* = \R_+^* \times \R.$$
	\end{rem}
	As a consequence of this remark, applying one main construction of \cite{Ma2018-3}, we get:
 \begin{Theorem} \label{th:hol}
 	$M_n(\R_+^*)$ projects to connections on $\Delta_n \times (\R_+^* \times \R)$ discretized by their holonomies, and the kernel of this projection is the space of pairwise comparisons matrices  $(a_{i,j})_{(i,j)\in \N_n^2} \in M_n(\R_+^*)$ such that 
 	\begin{equation*}
 	\forall (i,j)\in \N^2,	 i \neq j \Rightarrow a_{i,j}= 1  .
 	\end{equation*}
 \end{Theorem}
 
 We call \emph{pure pairwise comparisons matrix} a (maybe non reciprocal) matrix $(a_{i,j})_{(i,j)\in \N_n^2} \in M_n(\R_+^*)$ such that $\forall i \in \N_n, a_{i,i} = 1.$
 \begin{rem}
 	We may have made the choice to consider in our study only pure pairwise comparisons matrices. This would have given a more elegant formulation of last theorem, fully identifying discretized connections with the space of pairwise comparisons that we would have considered. However, for the sake of completeness of the framework that we propose here in view of modelizations, we feel the need to consider the assumption $a_{i,i}\neq 1$ as non void for applications. An example is briefly developped in the appendix, before follow-up investigations.
 \end{rem}
 \subsection{Inconsistency on $M_n(\R_+^*)$}
 Applying Theorem \ref{th:hol}, one can define consistent (maybe) non-reciprocal pairwise comparisons matrices by the discretized flat connections, the following way:
 \begin{Definition}
 	Let $A \in M_n(\R_+^*).$ Then the (maybe non reciprocal) pairwise comparisons is consistent if it is a pure pairwise comparisons matrix that projects to a flat discretized connection on $\Delta_n \times (\R_+^* \times \R).$ This space is noted by $CM_{n}(\R_+^*).$
 \end{Definition}
 
 This definition clearly extends the notion of consistent (and reciprocal) pairwise comarisons matrix. In the sequel, we shall make no difference in the notations between $CM_n(\R_+^*)$ as a set of pairwise comparions matrices with coefficients in $G$ and as a set of non-reciprocal pairwise comparisons in $M_n(\R_+^*)$ because of the bijection mentioned before between the two subsets, and we can define
 	\begin{Definition}
 	A \textbf{inconsistency index} is a map $ii:M_n(\R_+^*) \rightarrow \R_+$ such that $$ii^{-1}(0) = CM_n(\R_+^*).$$
 	A \textbf{inconsistency indicator} is an reciprocity index with values in $[0,1].$
 \end{Definition} 
  
  This definition clearly defines inconsistency indexes (resp. indicators) $ii$ on $M_n(\R_+^*)$ that restrict to inconsistency indexes (resp. indicators) on $PC_n$ along the lines of Definition \ref{d:ii}.
  \begin{example}
  	The discretized Yang-Mills functional along the lines of \cite{Ma2018-2} and for $G = T{\R_+^*}$ determines an inconsistency index that extends the known one in $PC_n.$
  \end{example}
  \begin{example}
  	Let $d$ be a distance on $M_n(\R_+^*).$ Then, for any $\gamma \in \R_+^*,$ the map $$ A \in M_n(\R_+^*) \mapsto \left(d(A,CM_n(\R_+^*))\right)^\gamma$$
  	is an inconsistency index.
  \end{example}
 
 One can define then a total index from an inconsistency index and a non reciprocity index.Let us finish by the following (again, direct) proposition:
 
 \begin{Proposition}[Geometric mean method in non-reciprocal setting] \label{nrgmm}
 	Let us note by $\pi':PC_n(G)\rightarrow PC_n(G)$ the map defined by the following formulas: if $A=(a_{i,j})_{(i,j)\in \N_n^2} \in PC_n(G), $ with
 	$$\forall (i,j)\in \N_n^2, a_{i,j} = (c_{i,j},d_{i,j}) \in \R_+^* \times \R,$$ 
 	then the image
 	$$ B=(b_{i,j})_{(i,j)\in \N_n^2} = \pi'(A), \quad \forall (i,j)\in \N_n^2, b_{i,j} = (c'_{i,j},d'_{i,j}) \in \R_+^* \times \R$$
 	is such that:
 	$$ (c'_{i,j})_{(i,j)\in \N_n^2} = \pi_{CPC}\left((c_{i,j})_{(i,j)\in \N_n^2}\right)$$
 	and 
 	$$\forall (i,j) \in \N_n^2, \quad d'_{i,j}= \frac{1}{n}\left(\sum_{k = 1}^n d_{i,k} - \sum_{l=1}^{n} d_{l,j}\right).$$
 	Then 
 	\begin{itemize}
 		\item the restriction of $\pi'$ to $PC_n({\R_+^*})$ is equal to $\pi_{CPC}$, where $PC_n({\R_+^*})$ is understood as the set of matrices in $PC_{n}(G)$ which coefficients are such that $\R-$component equal to $0,$ 
 		\item and moreover, $\pi'$ is a smooth projection from $PC_n(G)$ to $CM_n(\R_+^*).$  
 	\end{itemize}
 	\end{Proposition} 
 \begin{proof}
 	The first point follows from the definition of $\pi'$ in the coefficients $c'_{i,j}.$ Analysing now the formula for the coefficient $d'_{i,j},$ we recgnize a geometric mean method for an additive pairwise comparisons matrix, see \cite{SKE2023} for a review. 
 	Therefore, the obtained matrix $B$ is a projection to consistent pairwise comparisons  matrices with coefficient in $G.$ 
 	\end{proof}
 	 Therefore, we can propose the following projection:
 	\begin{Definition}
 		$\pi_{MC} = \pi' \circ \Phi.$
 	\end{Definition}
 	\begin{Theorem}
 		The following diagram is commutative:
 		
 		\begin{tikzcd}
 			M_n(\R_+^*) \arrow[r, "\pi_{PC}"] \arrow[d, "\pi_{CM}"]
 			& PC_n \arrow[d, "\pi_{CPC}" ] \\
 			MC_n \arrow[r, "\pi_{PC}"]
 			&  CPC_n
 		\end{tikzcd}
 		
 	\end{Theorem}
 	\begin{proof}
 		By direct calculations of the coefficients of $\pi_{CPC} \circ \pi_{PC} (A)$ on one side and of $\pi_{PC} \circ \pi_{CM}(A)$ in the other side. 
 	\end{proof}
 \section{Random pairwise comparisons}
 \subsection{The settings}
 	Let us now consider coefficients of matrices in $M_n(\R_+^*)$ as random variables in a sample probability space $\Omega$ with values in $M_n(\R_+^*)$ and we call it random pairwise comparisons matrix. The marginal laws on each coefficient $a_{i,j}$ of a random pairwise comparisons matrix $(a_{i,j})_{(i,j)\in \N_n^2}$ determine the probability law of each coefficient. 
 	
 For the sake of motivation and in order to precise accessible situations that are relevant to this picture, before the more developped example in the appendix, we precise that typically the coefficients $a_{i,j}$ follow the following laws:
 \begin{itemize}
 	\item a Gaussian law on the \underline{multiplicative} Lie group $\R_+^*$ if the evaluations are honnestly made by experts
 	\item an uniform law withe respect to an interval on $\R_+^*$, with respect to the \underline{additive} structure, if the errors are made by human mind corrections or cut-off approximations made with respect to the limits of the measurement procedures and without real link to the problem.
 \end{itemize}
 These two extreme cases are simple consequences of the celebrated Benford's law \cite{B1938}, but the law of the random pairwise comparisons matrix $A$ can be anything as well as the laws of the coefficients $a_{i,j}.$
 
 \begin{rem}
We have here to remark that, by randomness, the use of non-reciprocal pairwise comparisons become natural, even in ``honest'' evaluations due to the natural approximation of the scores. 
 \end{rem}
 
 Let us note by $RM_n(\R_+^*)$ the space of random matrices and, extending the notations, the first letter $R$ will mean ``random''. For example, $RPC_n$ is the set of random pairwise comparisons matrices with support in $PC_n.$
 \begin{Proposition}
 	There is an inclusion map $M_n(\R_+^*)\rightarrow RM_n(\R_+^*)$ yay identifies $(a_{i,j})_{(i,j)\in \N_n^2}$ with $\bigotimes_{(i,j)\in \N_n^2} \delta_{a_{i,j}}$ (Dirac measures). 
 \end{Proposition}

 \subsection{Stochastic analogs to deterministic constructions}
 In this section and till the end of the paper, all the (finite dimensional, geometric) mappings are assumed measurable with respect to the considered measures or probabilities. This is authomatically fulfilled for continuous (and hence for smooth) mappings with sourve and target in finite dimensional manifolds.
 \subsubsection{Stochastic extension of projections}
 Let us consider the projections already defined: $\pi_{CPC},$ $\pi_{PC}$ and $\pi_{CM}$ and our remark is the following: since these mappings are measurable, they extend to mappings between the corresponding probability spaces. That is, we have three push-forward maps 
 $$ (\pi_{CPC})_* : RPC_n \rightarrow RCPC_n,$$
 $$ (\pi_{PC})_* : RM_n(\R_+^*) \rightarrow RPC_n$$
 and 
 $$ (\pi_{CM})_* : RM_n(\R_+^*) \rightarrow RCM_n$$
 and their right inverse (that is, roughky speaking, $\pi_* \circ \pi^* = Id$ in the adequate domain) given by the pull-back mappings:
  $$ (\pi_{CPC})^* : RCPC_n \rightarrow RPC_n,$$
  $$ (\pi_{PC})^* : RPC_n \rightarrow RM_n(\R_+^*) $$
  and 
  $$ (\pi_{CM})^* : RCM_n \rightarrow RM_n(\R_+^*)$$
 \subsubsection{Stochastic indexes}
 We can now given the corresponding notions of the the various indexes of indicators in a stochastic context. 
 	\begin{Definition}
 	A \textbf{stochastic reciprocity index} (or \textbf{sr-index} for short)is a map $sri:RM_n(\R_+^*) \rightarrow \R_+$ such that $$sri^{-1}(0) = RPC_n.$$ 
 	A \textbf{stochastic reciprocity indicator} is a sr-reciprocity index with values in $[0,1].$
 \end{Definition} 
 The same approach by the support of the probability is proposed for each of the two following notions: 
 	\begin{Definition}
 	A \textbf{stochastic inconsistency index} (or \textbf{si-index} for short) is a map $sii:RM_n(\R_+^*) \rightarrow \R_+$ such that $$sii^{-1}(0) = CM_n(\R_+^*).$$
 	A \textbf{inconsistency indicator} is an sr-index with values in $[0,1].$
 \end{Definition} 
 \begin{Definition}
 	An  \textbf{stochastic total index} (or \textbf{st-index} for short) is a map $sti:M_n(\R_+^*) \rightarrow \R_+$ such that $$sti^{-1}(0) = CPC_n.$$
 	A \textbf{stochastic total indicator} is a st-index index with values in $[0,1].$
 \end{Definition}
 
 \begin{Theorem}
 	Let $i$ be a reciprocity index (resp. inconsistency indicator, resp. total indicator). Then it defines $si$ a sr-index (resp. si-index, resp. st-index)
 	by the formula:
 	 $$ si: X \in RM_n(\R_+^*) \mapsto \int i(A) dX(A).$$
 	 Moreover, $si$ restricts to $i$ on Dirac random pairwise comparisons matrices.
 \end{Theorem}
 
 The proof is obvious while considering the support of the map $si.$
 \begin{rem}
 	We need an inconsistency \emph{indicator} in this formula because the map $i$ needs to be, in fact, in the class $L^1(dX)$ for any $X.$ This theorem extends to any \emph{bounded} inconsistency index, but we prefer here to use the ``normalized'' approach proposed in \cite{KMal2017} even if the resulting stochastic index $si$ will be very exceptionaly a stochastic indicator. Since $si$ is bounded, it will be possible to normalize it by scaling in applications if necessary.
 \end{rem}
 
 \subsection{Minimizing inconsitency: a sketch of two methods}
 We choose here to expose the main steps of methods that one cn use for finding a consistent, maybe non reciprocal, or a reciprocal, or a consistent and reciprocal PC matrix. For the reasons explained in the presentation of the organisation of this work, the aspectts related to problems of convergence, uniqueness, regularity are postponed to future works. 
 \subsubsection{Minimization of a stochastic index}
 	The main idea consists in reducing the random matrix $X$ to a random matrix with support in $CM_n,$ $PC_n$ or $CPC_n$ depending on the nature of the bounded index $i$ chosen, by minimizing the associated stochastic index
 $$si: X \in \mathbb{P}(M_n(\R_+^*)) \mapsto \int_{RM_n(\R_+^*)} i(A) dX(A).$$
 This minimization problem is posed on the convex set $\mathbb{P}(M_n(\R_+^*))$ where the extremels are Dirac type measures, that is, elements of $M_n(\R_+^*). $ Beside the necessary developments of adequate analysis on the space $M_n(\R_+^*)$, let us point out some issues that cannot be treated in a ``black box'' that technicaly solves the problem, but that envolves conceptual constraints for decision makers:  
 \begin{itemize}
 	\item This method can be adapted to the constraints imposed by the decision problem. Indeed, not only the choice of the indicator $i$, but also a restricted space of random PC matrices, can be left to the choice of the decision maker which makes this method quite flexible. 
 	\item Besides this flexibility, the regularity of the functional is itself a technical issue and we hope that we will be able to succeeed in our future works that will precise the necessary regularity of the stochastic index in order to have a numericaly robust method.
 	\item The matrix $Y$ obtained as a minimizer of $si$ is still random, unless a way to impose to find it on the border $`` \partial RM_n(\R_+^*) = M_n(\R_+^*)''$ is imposed. 
 \end{itemize}
 \subsubsection{Shortest distance and optimal transport}
 	Therefore, one has to address the problem of finding a (non random) pairwise comparisons matrix from a random one. A direct approach, will be discussed next in a geometric way, may consist in taking the expectation value of the matrix $A.$ But it seems to us a bit drastic, because the first goal of a consistencization remains in changing \emph{as few as possible} (in a way that has to be controled numericaly) the coefficients of the PC matrix. If this condition of minimal change is not guaranteed, an (extreme) dictatorial behaviour would consist in imposing a pre-determined consistent and reciprocal PC matrix to any problem. We leave to the reader the production of a concrete example that can be modelled with such approach.   
 	
 	Anyway, we have to propose an approach that controls the ``de-randomization'' of the random PC matrix. Let us recall that any space of probabilities is a compact metric space, and for probabilities on a Riemannian manifold, the Wasserstein distance $W_2$ related to optimal transport problems is an improved metric for which one has already a panel of properties at hand, see e.g. \cite{Vil}.
Therefore, we can propose the following procedure given random pairwise comparisons matrix $A:$
 \begin{enumerate}
 	\item choose a \emph{compact} subset $\mathcal{B}$ in $M_n(\R_+^*)$ such that $\int 1_B dA > 1 - \epsilon$ and normalize $1_B dA$ as well as the Riemannian metric to make the volume of $B$ with respect to the new Riemannian metric equal to 1. The values that are outside the set $B$ are typicaly non-representative, or meaningless, values for the modelized problem.    
 	\item Consider the target space $T$ of Dirac measures in the desired space ($CM,$ $PC$ or $CPC$) in $B.$
 	\item Find $C \in T$ such that $$W_2(A,C)=min_{A' \in T} W_2(A,A')$$ 
 \end{enumerate}
 
 Let us now analyze, again from the viewpoint of the decision maker, the proposed method:
 \begin{itemize}
 	\item This method furnishes a way to define a deterministic matrix with respect to the minimization of a distance. This approach is then conceptualy affordable.
 	\item Moreover, the decision process can be acheived by the production of weights $(w_i)_{i \in \N_n}$ which explicitely proceed to a ranking,
 	\item but this method does not depend on any index. This means that  the minimization of the stochastic index $$si: X \in \mathbb{P}(M_n(\R_+^*)) \mapsto \int_{RM_n(\R_+^*)} i(A) dX(A)$$ has to be anyway performed before this second method (on the space $RM_n$) or after this methode (under the constrat of a minimization on the space of Dirac-type measures) if the decision maker wants to get a ranking really appropriate for the problem under consideration.
 \end{itemize}
 \section{A formal geometric picture for random pairwise comparisons and geometric concerns}
 
 Let us start with a simplified picture. This is a bundle-like extension construction, but due to tthe fact that we did not precise which topology we prefer to define on the space of random PC matrices, and also because the probability space in general is not a manifold, in the best cases a metric space \cite{Pro} or a diffeological space \cite{Ma2020-3,GMW2023}, the full geometric or topological structures will remain non precised in this communication. We shall  only mention the following diagram:

 \begin{tikzcd}
 	RM_n(\R_+^*) \arrow[r, "\pi_{CPC,*}"] \arrow[d, "\mathbb{E}"]
 	& M_n(\R_+^*) \arrow[d, "\mathbb{E} " ] \\
 	M_n(\R_+^*) \arrow[r, "\pi_{CPC}"]
 	&  CPC_n
 \end{tikzcd}
 where $\mathbb{E}$ is the (multiplicative) expectation value. 

\begin{Proposition}
Last diagram is well defined and it commutes.
\end{Proposition}

\begin{proof}
	The proof is standard for additive pairwise comparisons, where the expectation value is taken in a standard way, and where the projection $\pi_{CPC}$ reads as a mean of the reciprocal coefficients followed by an arithmetic mean on off-diagonal coefficients. the multiplicative expectation value of the multiplicative PC matrix is the exponential of the expectation value of this additive PC matrix, while the geometric mean projection follows the same rules.  
\end{proof}

 Therefore, we have exhibited a split of the space of random pairwise comparisons matrix into a fiber bundle where
 \begin{enumerate}
 	\item The projection associates at each random matrix $A$ the matrix $\mathbb{E}(A)$ of the (multiplicative) expectation values of  its coefficitent, and the geometric mean method is, in the sense that we have explained, compatible with this fibration
 	\item Therefore the fibers can be identified with multiplicative random pairwise comparisons matrices with expectation values equal to $1$ where $1$ is the neutral element of $M_n(\R_+^*).$ It has a ``global section'' defined through Dirac measures. Each fiber can be heuristicaly identified as the stochastic variations with trivial expectation values of the given deterministic PC matrix on the base. 
 	\item Even if the base is a finite dimensioanal Lie group, the fibers are truely infinite dimensional and do not carry a structure of e.g. Banach manifold as in many known cases.
 \end{enumerate}
  \section{Outlook}
  The aim of this paper was to describe safely the basic necessary setting for random pairwise comparisons, in order to adapt it and interpret it in view of applications. After this description, we see four directions for future investigations.
  \begin{enumerate}
  	\item The case when the non reciprocal pairwise comparisons matrix is not pure, that is, when the diagonal coefficients are not neutral, needs a more deep treatise, maybe inspired by concepts in physics.
  	\item The two procedures for reduction of a stochastic index that we proposed have to be implemented, both in terms of optimal control theory. This requires to gather the known results that are related, and produce a proper algorithm which convergence is fully studied.
  	\item For this, the geometric structures (metric geometry, diffeology... ) of the ambient space of random matrices as well as the useful geometric objects for the implimentation of the algorithms already mentioned have to be precised. Moreover, we get another motivation from the fact that the two proposed types of structures, metric spaces and diffeologies, have not been deeply compared yet and such a study in a given example would be of theoretical interest.
  	\item Finally, we did not address in this paper the question of maximal inconsistency of a submatrix of a  random PC matrix.  
  \end{enumerate} 
 	\section*{ Appendix: One motivating example on exchange rates}
 	Let us consider three currencies A, B and C and let us explain what happens in the corresponding exchange rates.
 	
 	
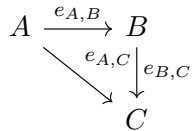
\begin{tikzcd}
 		A \arrow[rd, "e_{A,C}" ] \arrow[r, "e_{A,B}"] & B \arrow[d, "e_{B,C}"]\\
 		& C
 	\end{tikzcd}
 	
 	For almost any existing currency, one can see that
 	\begin{itemize}
 		\item The exchange rates $e_{K,L},$ for $(K,L) \in \{A,B,C\}^2,$ truely gives a reciprocal ratio between the currencies, i.e. $ \forall (K,L) \in \{A,B,C\}^2, e_{L,K}=e_{K,L}^{-1}.$
 		\item a loop that changes $A$ to $B,$ $B$ to $C$ and $C$ to $A$ may give an inconsistency $$e_{A,B}e_{B,C}e_{C,A} \neq e_{AA} = 1.$$ 
 	\end{itemize}
 Moreover, any individual who wants to exchange e.g. $A$ for $B$ needs to contact a currency exchange agency which applies charges and taxes to exchanges, which implies
 $$ \forall (K,L) \in \{A,B,C\}^2, e_{L,K} \neq e_{K,L}^{-1},$$
 and storing currencies on a bank account generates costs, which reads as $$\forall L \in \{A,B,C\}, e_{L,L}\neq 1.$$ These two items are non reciprocal pairwise comparisons.


\begin{thebibliography}{150}
	\bibitem{AZ1990} Albeverio, S.; Zegarlinski, B.; Construction of convergent simplicial approximations of quantum fields on Riemannian manifolds \textit{Comm. Mat. Phys.} \textbf{132} 39-71 (1990)
	\bibitem{BR2022} Bartl, D.; Ramík, J.; A new algorithm for computing priority vector of
	pairwise comparisons matrix with fuzzy elements. \emph{ Information Sciences},
	{\bf 615} 103?117 (2022).
	\bibitem{B1938} Benford, F.; The law of anomalous numbers. \emph{Proc. Am. Philos. Soc.} {\bf 78} no 4, 551?572 (1938)
	\bibitem{BR2008} Boz\'oki, S.;  Rapcs\'ak, T.;  On
	Saaty's and
	Koczkodaj's inconsistencies
	of
	pairwise
	comparison
	matrices;
	\textit{J.
		Glob. Optim.}
	\textbf{42} no.2
	(2008)
	157-175 
	\bibitem{6} Colomer, J-M.; Ramon Llull: from ?Ars electionis? to social choice theory.
	\emph{Social Choice and Welfare} {\bf 40} no2), 317?328 (2011).
	\bibitem{CW1985} Crawford, R.; Williams, C.; A note on the analysis of subjective judgement matrices. \emph{Journal of Mathematical Psychology}, {\bf 29} 387--405 (1985)
	\bibitem{10} A. Darko, A.;  Chan,  A. P. C.; Ameyaw, E. E.; Owusu,  E. K.;  P\"arn, E.; D. J.
	Edwards,  D. J.; Review of application of analytic hierarchy process (AHP) in construction. \emph{International Journal of Construction Management} {\bf 19} no5, 436?
	452 (2019)
	\bibitem{FV2012} Farinelli S., Vasquez S. Gauge invariance, geometry and arbitrage. \emph{J. Invest. Strategies}. {\bf 1} no 2, 2366 (2012).
	\bibitem{F2008} F\"ul\"op, J.; A method for approximating pairwise comparisons matrices by consistent matrices \textit{J. Global Optimization} \textbf{42}, 423-442 (2008)
	\bibitem{14} Gacula Jr., M.C.;  Singh, J.; \emph{Statistical Methods in Food and Consumer Research}. Academic Press (1984).
	\bibitem{GMW2023} Goldammer, N.; Magnot, J-P.; Welker, K.; On diffeologies from infinite dimensional geometry to PDE constrained optimization. To appear in \emph{Contemporary Mathematics}, volume dedicated to the proceedings of the special session ``Recent advances in diffeologies and their applications'', AMS-EMS-SMF congress, Grenoble, France, July 2022.
		\bibitem{Gro} 
	Gromov, M.; \emph{ Metric structures in Riemannian and non-Riemannian spaces}, 2nd ed. Birkauser, (1997).
	\bibitem{18} Hyde, R.A.; Davis, K.; Military applications of the analytic hierarchy
	process. \emph{International Journal of Multicriteria Decision Makin}g, {\bf 2} no3), 267
	(2012)
	
	\bibitem{I2000} Illinski, K.; Gauge geometry of financial markets. \emph{J. Phys. A, Math. Gen.} {\bf 33} 5?14 (2000).
\bibitem{KN} Kobayashi, S.; Nomizu, K. (1963-1969) {`\it Foundations of Differential Geometry I and II} Wiley classics library ISBN:978-0471157335 and   978-0471157328
	\bibitem{K1993} Koczkodaj, W.W. (1993) A new definition
	of consistency
	of pairwise
	comparisons,\textit{Math. Comput.
		Modelling} {\bf 8}
	 79-84 (1993).
	\bibitem{KO1997} Koczkodaj, W.; Orlowski, M.; An orthogonal basis for computing a
	 consistent approximation to a pairwise comparisons matrix. \emph{Computers
	 and Mathematics with Applications} {\bf 34} no10, 41?47 (1997)
		\bibitem{KMal2017}	 Koczkodaj, W.W.;  Magnot, J-P.;  Mazurek, J.; Peters, J.F.;
	Rakhshani, H.; Soltys, M.;  Strza lka, D.; Szybowski, J.;   Tozzi, A. (2017) On normalization of inconsistency indicators
	in pairwise comparisons \textit{Int. J. Approx. Reasoning} \textbf{86}  73-79. https://doi.org/10.1016/j.ijar.2017.04.005
	\bibitem{KSS2020} Koczkodaj, W.W.; Smarzewski, R.; Szybowski, J. (2020) On Orthogonal Projections on the Space of Consistent Pairwise Comparisons Matrices \emph{Fundamenta Informaticae}, {\bf 172}, no. 4  379-397 https://doi.org/10.1016/j.ijar.2017.04.005
	\bibitem{KSW2016}
	Koczkodaj, W.W.; Szybowski, J.; Wajch, E.,
	Inconsistency indicator maps on groups for pairwise comparisons, International Journal of Approximate Reasoning, 69(2) 81-90, 2016.
	\bibitem{KS2014} 
	Koczkodaj, W.W.; Szwarc, R. (2014)
	Axiomatization of Inconsistency Indicators for Pairwise Comparisons, \textit{Fundam. Inform.} 
\textbf{132} no 4, 485-500 (2014)
\bibitem{KGDW2015} Ku?akowski, K. Grobler-D?bska, K; W?s, J.; Heuristic rating estima-
tion: geometric approach. \emph{Journal of Global Optimization}, {\bf 62} no3, 529?543
(2015).
\bibitem{27} K. Kulakowski, J. Mazurek, and M. Strada. On the similarity between
ranking vectors in the pairwise comparison method. \emph{Journal of the Opera-
tional Research Society}, {\bf 0} 1?10 (2021).
\bibitem{LLA2012} Lahby, M.; Leghris, C.; Adib, A.; Network Selection Decision based on handover history
in Heterogeneous Wireless Networks; \textit{International 
	Journal of 
	Computer 
	Science and 
	Telecommunications} {\bf 3} no 2,  (2012) 21-25
	\bibitem{32} Liberatore, M.J.; Nydick, R. L.; Group decision making in higher education using the analytic hierarchy process. \emph{ Research in Higher Education},
	{\bf 38} no5, 593?614 (1997).
\bibitem{Mac95} Mack, G.; Gauge theory of things alive \textit{ Nucl.Phys. B, Proc.Suppl.} \textbf{42} 923-925 (1995)
	\bibitem{Ma2018-1} Magnot, J-P.; A mathematical bridge between discretized gauge theories in quantum physics and approximate reasoning in pairwise comparisons \textit{Adv. Math. Phys.} \textbf{2018} (2018), Article ID 7496762, 5 pages https://doi.org/10.1155/2018/7496762
		\bibitem{Ma2018-2} Magnot, J-P.;  Remarks on a new possible discretization scheme for gauge theories. 
	\emph{Int. J. Theor. Phys.} {\bf 57} no. 7, 2093-2102 (2018)
	\bibitem{Ma2018-3} Magnot, J-P. ; On mathematical structures on pairwise comparisons matrices with coefficients in a group arising from quantum gravity {\it Helyion} {\bf 5} (2019) e01821 https://doi.org/10.1016/j.heliyon.2019.e01821
	\bibitem{Ma2020-3} Magnot, J-P.; On the differential geometry of numerical schemes and weak solutions of
	functional equations. \emph{Nonlinearity} {\bf 33}, no. 12, 6835-6867 (2020) 
	\bibitem{MMC2023} Magnot, J-P.; Mazurek, J.; Cernanova, V.; A gradient method for inconsistency reduction of pairwise comparisons matrices \emph{Int. J. Approx. Reasonning} {\bf 152} 46-58 (2023)
	\bibitem{37} Peterson, G.L.; Brown, T.C.; Economic valuation by the method of paired comparison, with emphasis on evaluation of the transitivity axiom.
	\emph{Land Economics} 240-261 (1998)
	\bibitem{Pro} Prokhorov, Yu.V.; Convergence of random processes and limit theorems in probability
	theory. \emph{Theor. Prob. Appl.} {\bf 1} 157-214 (1956)
	
	

	\bibitem{SSSA} Sen, S.; Sen, S.; Sexton, J.C.; Adams, D.H.; A geometric discretisation scheme applied to the Abelian Chern-Simons theory \textit{Phys.Rev. E} \textbf{61}  3174-3185 (2000)
\bibitem{RV2014} Rovelli, C.; Vidotto, F. (2014) {\it Covariant loop quantum gravity} Cambridge university press https://doi.org/10.1017/CBO9781107706910
	\bibitem{S1977} Saaty, T.; A scaling methods for priorities in hierarchical structures; \textit{J. Math. Psychol.} \textbf{15} (1977) 234-281
	\bibitem{41} Sasaki, Y.; Strategic manipulation in group decisions with pairwise comparisons: A game theoretical perspective. \emph{ European Journal of Operational
	Research} {\bf 304} no3, 1133?1139 (2023).
	\bibitem{42} Taylor, A.D.; \emph{Social Choice and the Mathematics of Manipulation}. Outlooks. Cambridge University Press (2005).
	\bibitem{43} Thurstone, L.L.; The Method of Paired Comparisons for Social Values. \emph{Journal of Abnormal and Social Psychology}, pages 384?400 (1927).
	\bibitem{45}  Urbaniec, M.; So?tysik, M.; Prusak, A.; Ku?akowski, K.; 
 Wojnarowska, M.; Fostering sustainable entrepreneurship by busi-
	ness strategies: An explorative approach in the bioeconomy. {Business Strategy and the Environment} {\bf 31} no1, 251?267 (2022).
	\bibitem{Seng2011} Sengupta, A.N. (2011)
	Yang-Mills in Two Dimensions and Chern-Simons in Three
	, in
\textit{	Chern-
	Simons  Theory:   20  years  after},  Editors  Jorgen  Ellegaard  Anderson,  Hans  U.  Boden,  Atle
	Hahn,  and  Benjamin  Himpel.   AMS/IP  Studies  in  Advanced  Mathematics  (pp.   311-320).
	\bibitem{S44} Sengupta, B.; Tozzi, A.; Cooray, G.K.; Douglas, P.K.; Friston, K.J.; Towards a Neuronal
	Gauge Theory. \emph{ PLoS Biol.} {\bf 14} no 3, e1002400 (2016).
	\bibitem{SKE2023} Szybowski, J; Kulakowski, K.; Ernst, S.; Almost optimal manipulation of a pair of alternatives \texttt{ArXiv:2304.03060}

\bibitem{Vil} Villani, C.; \emph{Optimal transport, old and new} Springer (2006)

\bibitem{W2019} Wajch, E., From pairwise comparisons to consistency with respect to a group operation and Koczkodaj's metric. \emph{Int. J. Approx. Reason.} {\bf 106} 51?62 (2019).

\bibitem{Wh} 
Whitney, H. (1957) \emph{Geometric Integration Theory}, Princeton University Press, Princeton, , J.NJ.




\bibitem{Y1999} Young K. Foreign exchange markets as a lattice gauge theory. \emph{Am. J. Phys.} {\bf 67} no 10, 862?868 (1999).


\bibitem{48} Yuan, R.; Wu, Z.; Tu, J.; Large-scale group decision-making with incomplete fuzzy preference relations: The perspective of ordinal consistency.
\emph{Fuzzy Sets and Systems} {\bf 454} 100?124 (2023).
\end{thebibliography}
\end{document}